\documentclass[10pt]{article}

\usepackage{amsthm, amsmath, amsfonts, amssymb, latexsym}
\usepackage[utf8]{inputenc}
\usepackage[english]{babel}
\usepackage[hidelinks]{hyperref}
\usepackage{orcidlink}
\usepackage{graphicx}
\usepackage{pict2e}

\usepackage[backend=biber,style=trad-unsrt,doi=true,url=false,isbn=false]{biblatex}
\theoremstyle{plain}
\newtheorem{theorem}{Theorem}

\theoremstyle{definition}

\newtheorem{example}{Example}
\theoremstyle{remark}

\begin{document}

    \title{Branching $k$-path vertex cover of forests}
    \author{Mikhail Makarov\orcidlink{0009-0004-3604-4547}\footnote{mikhail.makarov.math@gmail.com}}
    \date{}
    \maketitle

    \begin{abstract}
        We define a set $P$ to be a \emph{branching $k$-path vertex cover} of an undirected forest $F$ if all leaves and isolated vertices (vertices of degree at most $1$) of $F$ belong to $P$ and every path on $k$ vertices (of length $k-1$) contains either a branching vertex (a vertex of degree at least $3$) or a vertex belonging to $P$. We define the \emph{branching $k$-path vertex cover number} of an undirected forest $F$, denoted by $\psi_b(F,k)$, to be the number of vertices in the smallest branching $k$-path vertex cover of $F$. These notions for a rooted directed forest are defined similarly, with natural adjustments. We prove the lower bound $\psi_b(F,k) \geq \frac{n+3k-1}{2k}$ for undirected forests, the lower bound $\psi_b(F,k) \geq \frac{n+k}{2k}$ for rooted directed forests, and that both of them are tight.
    \end{abstract}

    \section{Introduction}
        Throughout the article, we follow the standard convention that a \emph{path} in a graph is defined to have all its vertices distinct. Also, we follow the standard convention that a \emph{leaf} in an undirected graph is defined as a vertex of degree exactly $1$, and a \emph{leaf} in a directed graph is defined as a vertex of out-degree $0$. Thus, leaves in an undirected graph do not include isolated vertices, while leaves in a directed graph do.
        
        We define a set $P$ to be a \emph{branching $k$-path vertex cover} of an undirected forest $F$ if all leaves and isolated vertices (vertices of degree at most $1$) of $F$ belong to $P$ and every path on $k$ vertices (of length $k-1$) contains either a branching vertex (a vertex of degree at least $3$) or a vertex belonging to $P$. We define the \emph{branching $k$-path vertex cover number} of an undirected forest $F$, denoted by $\psi_b(F,k)$, to be the number of vertices in the smallest branching $k$-path vertex cover of $F$.

        We investigate the problem of establishing a lower bound on $\psi_b(F,k)$ in terms of $n$ and $k$. We first solve that problem for the directed case (for rooted directed forests), and then derive the undirected case from the directed one.

        We say that a directed graph is a \emph{rooted directed forest} if the undirected graph obtained from it by removing the orientations of the edges is a forest in which each connected component has a designated vertex, called the \emph{root}, such that all arcs corresponding to the edges of that connected component are oriented away from that root.
        
        The definition of a branching $k$-path vertex cover for the directed case is slightly different. Naturally, it involves a directed path instead of an undirected one, a branching vertex is defined as a vertex of out-degree at least $2$, but, most noticeably, we do not need to separately require isolated vertices to belong to $P$, because leaves, defined as vertices of out-degree $0$, already include isolated vertices. Specifically, we define a set $P$ to be a \emph{branching $k$-path vertex cover} of a rooted directed forest $F$ if $P$ is a subset of the vertices of $F$ such that all leaves (vertices of out-degree $0$) belong to $P$ and every directed path on $k$ vertices (of length $k-1$) contains either a branching vertex (a vertex of out-degree at least $2$) or a vertex belonging to $P$. We define the \emph{branching $k$-path vertex cover number} of a rooted directed forest $F$, denoted by $\psi_b(F,k)$, to be the number of vertices in the smallest branching $k$-path vertex cover of $F$.

        The condition related to the one in our definition of the branching $k$-path vertex cover, that every path on $k$ vertices contains a vertex belonging to $P$ (without the alternative of containing a branching vertex and without the condition on leaves), was previously studied for undirected graphs as the problem of \emph{$k$-path vertex cover}~\cite{tu_survey_k-path_vertex_cover} and the associated \emph{$k$-path vertex cover number} of a graph $G$, which we denote by $\psi(G,k)$.
                
        Informally, the idea behind our branching variant is that more branching vertices result in more leaves, and the leaves are required to belong to $P$. So, although a branching vertex relaxes the condition for paths containing it by allowing them not to contain vertices from $P$, it still, in a way, forces more vertices to belong to $P$.

        As examples, we calculate $\psi_b(F,k)$ and $\psi(F,k)$ for the following two undirected trees.
        
        \begin{example}\label{example:with_7_vertices}
            \begin{figure}[htb]
            \centering
            \setlength{\unitlength}{0.4mm}
            \begin{picture}(252,67)(14,27)
                \put(50,50){\circle*{3}}
                
                \put(50,70){\circle*{3}}
                \put(50,90){\circle*{3}}
                \Line(50,50)(50,70)
                \Line(50,70)(50,90)
                
                \put(32.68,40){\circle*{3}}
                \put(15.36,30){\circle*{3}}
                \Line(50,50)(32.68,40)
                \Line(32.68,40)(15.36,30)
                
                \put(67.32,40){\circle*{3}}
                \put(84.64,30){\circle*{3}}
                \Line(50,50)(67.32,40)
                \Line(67.32,40)(84.64,30)

                \put(140,50){\circle*{3}}
                
                \put(140,70){\circle*{3}}
                \put(140,90){\circle*{3}}
                \put(140,90){\circle{7}}
                \Line(140,50)(140,70)
                \Line(140,70)(140,90)
                
                \put(122.68,40){\circle*{3}}
                \put(105.36,30){\circle*{3}}
                \put(105.36,30){\circle{7}}
                \Line(140,50)(122.68,40)
                \Line(122.68,40)(105.36,30)
                
                \put(157.32,40){\circle*{3}}
                \put(174.64,30){\circle*{3}}
                \put(174.64,30){\circle{7}}
                \Line(140,50)(157.32,40)
                \Line(157.32,40)(174.64,30)

                \put(230,50){\circle*{3}}
                \put(230,50){\circle{7}}
                
                \put(230,70){\circle*{3}}
                \put(230,90){\circle*{3}}
                \Line(230,50)(230,70)
                \Line(230,70)(230,90)
                
                \put(212.68,40){\circle*{3}}
                \put(195.36,30){\circle*{3}}
                \Line(230,50)(212.68,40)
                \Line(212.68,40)(195.36,30)
                
                \put(247.32,40){\circle*{3}}
                \put(264.64,30){\circle*{3}}
                \Line(230,50)(247.32,40)
                \Line(247.32,40)(264.64,30)
            \end{picture}
            \caption{The tree $T_1$ from Example~\ref{example:with_7_vertices} (left), its branching $3$-path vertex cover (middle), and its (classical) $3$-path vertex cover (right).}
            \label{figure:example_with_7_vertices}
        \end{figure}
            Consider the tree shown in Figure~\ref{figure:example_with_7_vertices} on the left and denote it by $T_1$. It has $7$ vertices, $6$ edges, $3$ leaves, and $1$ branching vertex in the center. It is easy to see that $\psi_b(T_1,3)=3$ because any branching $3$-path vertex cover must contain the $3$ leaves and the $3$ leaves are sufficient, as shown in Figure~\ref{figure:example_with_7_vertices} in the middle, since the branching vertex covers the remaining paths on $3$ vertices that are not covered by the leaves (in fact, the branching vertex covers \emph{all} paths on $3$ vertices, both those that are already covered by the leaves and those that are not covered by the leaves). On the other hand, $\psi(T_1,3)=1$ because the single vertex in the center covers all paths on $3$ vertices, as shown in Figure~\ref{figure:example_with_7_vertices} on the right. So, we have $\psi_b(T_1,3)=3>1=\psi(T_1,3)$.
        \end{example}

        \begin{example}\label{example:with_16_vertices}
            \begin{figure}[htb]
                \centering
                \setlength{\unitlength}{0.3mm}
                \begin{picture}(386,171)(7,54)
                    \put(60,95){\circle*{3}}
                    \put(60,125){\circle*{3}}
                    \put(60,155){\circle*{3}}
                    \put(60,185){\circle*{3}}
                    \Line(60,95)(60,125)
                    \Line(60,125)(60,155)
                    \Line(60,155)(60,185)
                    
                    \put(42.68,85){\circle*{3}}
                    \put(25.36,75){\circle*{3}}
                    \put(8.04,65){\circle*{3}}
                    \Line(60,95)(42.68,85)
                    \Line(42.68,85)(25.36,75)
                    \Line(25.36,75)(8.04,65)
                    
                    \put(77.32,85){\circle*{3}}
                    \put(94.64,75){\circle*{3}}
                    \put(111.96,65){\circle*{3}}
                    \Line(60,95)(77.32,85)
                    \Line(77.32,85)(94.64,75)
                    \Line(94.64,75)(111.96,65)
                    
                    \put(42.68,195){\circle*{3}}
                    \put(25.36,205){\circle*{3}}
                    \put(8.04,215){\circle*{3}}
                    \Line(60,185)(42.68,195)
                    \Line(42.68,195)(25.36,205)
                    \Line(25.36,205)(8.04,215)
                    
                    \put(77.32,195){\circle*{3}}
                    \put(94.64,205){\circle*{3}}
                    \put(111.96,215){\circle*{3}}
                    \Line(60,185)(77.32,195)
                    \Line(77.32,195)(94.64,205)
                    \Line(94.64,205)(111.96,215)

                    \put(195,95){\circle*{3}}
                    \put(195,125){\circle*{3}}
                    \put(195,155){\circle*{3}}
                    \put(195,185){\circle*{3}}
                    \Line(195,95)(195,125)
                    \Line(195,125)(195,155)
                    \Line(195,155)(195,185)
                    
                    \put(177.68,85){\circle*{3}}
                    \put(160.36,75){\circle*{3}}
                    \put(143.04,65){\circle*{3}}
                    \Line(195,95)(177.68,85)
                    \Line(177.68,85)(160.36,75)
                    \Line(160.36,75)(143.04,65)
                    
                    \put(212.32,85){\circle*{3}}
                    \put(229.64,75){\circle*{3}}
                    \put(246.96,65){\circle*{3}}
                    \Line(195,95)(212.32,85)
                    \Line(212.32,85)(229.64,75)
                    \Line(229.64,75)(246.96,65)
                    
                    \put(177.68,195){\circle*{3}}
                    \put(160.36,205){\circle*{3}}
                    \put(143.04,215){\circle*{3}}
                    \Line(195,185)(177.68,195)
                    \Line(177.68,195)(160.36,205)
                    \Line(160.36,205)(143.04,215)
                    
                    \put(212.32,195){\circle*{3}}
                    \put(229.64,205){\circle*{3}}
                    \put(246.96,215){\circle*{3}}
                    \Line(195,185)(212.32,195)
                    \Line(212.32,195)(229.64,205)
                    \Line(229.64,205)(246.96,215)

                    \put(143.04,65){\circle{7}}
                    \put(246.96,65){\circle{7}}
                    \put(143.04,215){\circle{7}}
                    \put(246.96,215){\circle{7}}

                    \put(330,95){\circle*{3}}
                    \put(330,125){\circle*{3}}
                    \put(330,155){\circle*{3}}
                    \put(330,185){\circle*{3}}
                    \put(312.68,85){\circle*{3}}
                    \put(295.36,75){\circle*{3}}
                    \put(278.04,65){\circle*{3}}
                    \put(347.32,85){\circle*{3}}
                    \put(364.64,75){\circle*{3}}
                    \put(381.96,65){\circle*{3}}
                    \put(312.68,195){\circle*{3}}
                    \put(295.36,205){\circle*{3}}
                    \put(278.04,215){\circle*{3}}
                    \put(347.32,195){\circle*{3}}
                    \put(364.64,205){\circle*{3}}
                    \put(381.96,215){\circle*{3}}

                    \Line(330,95)(330,125)
                    \Line(330,125)(330,155)
                    \Line(330,155)(330,185)
                    
                    \Line(330,95)(312.68,85)
                    \Line(312.68,85)(295.36,75)
                    \Line(295.36,75)(278.04,65)
                    
                    \Line(330,95)(347.32,85)
                    \Line(347.32,85)(364.64,75)
                    \Line(364.64,75)(381.96,65)
                    
                    \Line(330,185)(312.68,195)
                    \Line(312.68,195)(295.36,205)
                    \Line(295.36,205)(278.04,215)
                    
                    \Line(330,185)(347.32,195)
                    \Line(347.32,195)(364.64,205)
                    \Line(364.64,205)(381.96,215)

                    \put(330,125){\circle{7}}  
                    \put(312.68,85){\circle{7}}  
                    \put(347.32,85){\circle{7}}  
                    \put(312.68,195){\circle{7}}  
                    \put(347.32,195){\circle{7}}  
                    
                    \put(330,125){\oval(16,70)}
                    \put(295.36,75){\rotatebox{30}{\oval(55,16)}}
                    \put(364.64,75){\rotatebox{-30}{\oval(55,16)}}
                    \put(295.36,205){\rotatebox{-30}{\oval(55,16)}}
                    \put(364.64,205){\rotatebox{30}{\oval(55,16)}}
                \end{picture}
                \caption{The tree $T_2$ from Example~\ref{example:with_16_vertices} (left), its branching $3$-path vertex cover (middle), and its (classical) $3$-path vertex cover (right).}
                \label{figure:example_with_16_vertices}
            \end{figure}
            Consider the tree shown in Figure~\ref{figure:example_with_16_vertices} on the left and denote it by $T_2$. It has $16$ vertices, $15$ edges, $4$ leaves, and $2$ branching vertices. It is easy to see that $\psi_b(T_2,3)=4$ because any branching $3$-path vertex cover must contain the $4$ leaves and the $4$ leaves are sufficient, as shown in Figure~\ref{figure:example_with_16_vertices} in the middle, since the branching vertices cover the remaining paths on $3$ vertices that are not covered by the leaves. On the other hand, $\psi(T_2,3)=5$ because we can select $5$ vertex-disjoint paths on $3$ vertices each in $T_2$, as shown in Figure~\ref{figure:example_with_16_vertices} on the right, each of which must contain at least $1$ vertex of the cover, and we can verify directly that the $5$ vertices shown in Figure~\ref{figure:example_with_16_vertices} on the right are sufficient. So, we have $\psi_b(T_2,3)=4<5=\psi(T_2,3)$.
        \end{example}

        The two examples demonstrate that both inequalities $\psi_b(F,k)>\psi(F,k)$ and $\psi_b(F,k)<\psi(F,k)$ are possible.

        Unlike the classical $k$-path vertex cover, the branching variant is meaningful only for forests or, possibly, for some graphs close to forests or for some very specific classes of graphs. In arbitrary graphs, there could be many branching vertices and very few leaves, which would make the branching $k$-path vertex cover number very small. For example, in $K_n$, all vertices are branching, but there are no leaves, so we have $\psi_b(K_n,k)=0$.

        There is a trivial upper bound $\psi_b(F,k) \leq n$ for both undirected forests and rooted directed forests, and it is tight, as it is attained on the forest consisting of $n$ isolated vertices. So, the main non-trivial problem is about the lower bounds.

    \section{The directed case}
        \begin{theorem}\label{theorem:branching_k-path_vertex_cover_rooted_directed_forest_lower_bound}
            Let $F$ be a rooted directed forest on $n \geq 1$ vertices. Let $k \geq 2$ be a natural number. Then $\psi_b(F,k) \geq \frac{n+k}{2k}$. That lower bound is tight (when the expression in the lower bound is an integer).
        \end{theorem}
        \begin{proof}
            Let $P$ be an arbitrary branching $k$-path vertex cover of $F$. We prove the lower bound $|P| \geq \frac{n+k}{2k}$ using induction on $n$.
                        
            The base case of the induction is for $1 \leq n \leq k$. It is well-known that a rooted directed forest always contains at least one leaf. That leaf must belong to $P$. Therefore, we have $|P| \geq 1 \geq \frac{n+k}{2k}$.
                        
            Suppose that the statement holds for all rooted directed forests with fewer than $n$ vertices and consider it for $n$ vertices, where $n \geq k+1$. It is well-known that a rooted directed forest always contains at least one leaf. For each leaf, we construct a directed path on at most $k$ vertices ending at that leaf as follows. We start with a leaf $v_1$ in $F$, and go from it through its ancestors by reversing the unique directed path from the root to the leaf $v_1$, taking the vertices $v_1$, $v_2$, \ldots on this path. We stop when, in the directed path $v_q \ldots v_1$ of the taken vertices, the last taken vertex $v_q$ either is a root, or has a parent $u$ that is a branching vertex, or has a parent $u$ that belongs to $P$. By construction, since the process did not stop at the vertices $v_1$, \ldots, $v_{q-1}$, their parent vertices $v_2$, \ldots, $v_q$ do not belong to $P$ and are not branching vertices. The vertex $v_1$ is also not a branching vertex because it is a leaf. So, none of the vertices in the directed path $v_q \ldots v_1$ are branching vertices. And the only vertex in that directed path $v_q \ldots v_1$ that can belong to $P$ is the vertex $v_1$, which does belong to $P$ because it is a leaf. Therefore, exactly $1$ vertex in the directed path $v_q \ldots v_1$ belongs to $P$. If $q \geq k+1$, then the directed path $v_{k+1} \ldots v_2$ on $k$ vertices does not contain branching vertices or vertices belonging to $P$, which contradicts the condition of the branching $k$-path vertex cover. Therefore, we must have $q \leq k$.

            \begin{figure}[htb]
                \centering
                \setlength{\unitlength}{0.4mm}
                \begin{picture}(235,94)(26,25)
                    \put(50,95){\circle*{3}}
                    \put(50,75){\circle*{3}}
                    \put(50,55){\circle*{3}}
                    \put(50,35){\circle*{3}}
                    \put(50,95){\vector(0,-1){19}}
                    \put(50,75){\vector(0,-1){19}}
                    \put(50,55){\vector(0,-1){19}}
                    \put(39,95){\makebox(0,0)[r]{$v_q$}}
                    \put(62,95){\makebox(0,0)[l]{root}}
                    \put(39,35){\makebox(0,0)[r]{$v_1$}}
                    \put(50,35){\circle{7}}
                    \put(50,65){\oval(16,80)}
                    
                    \put(150,115){\circle*{3}}
                    \put(150,95){\circle*{3}}
                    \put(150,75){\circle*{3}}
                    \put(150,55){\circle*{3}}
                    \put(150,35){\circle*{3}}
                    \put(150,115){\vector(0,-1){19}}
                    \put(150,95){\vector(0,-1){19}}
                    \put(150,75){\vector(0,-1){19}}
                    \put(150,55){\vector(0,-1){19}}
                    \put(150,115){\vector(1,-1){19.2}}
                    \put(170,95){\circle*{3}}
                    \put(145,115){\makebox(0,0)[r]{$u$}}
                    \put(139,95){\makebox(0,0)[r]{$v_q$}}
                    \put(139,35){\makebox(0,0)[r]{$v_1$}}
                    \put(150,35){\circle{7}}
                    \put(150,65){\oval(16,80)}
                    
                    \put(250,115){\circle*{3}}
                    \put(250,95){\circle*{3}}
                    \put(250,75){\circle*{3}}
                    \put(250,55){\circle*{3}}
                    \put(250,35){\circle*{3}}
                    \put(250,115){\vector(0,-1){19}}
                    \put(250,95){\vector(0,-1){19}}
                    \put(250,75){\vector(0,-1){19}}
                    \put(250,55){\vector(0,-1){19}}
                    \put(243,115){\makebox(0,0)[r]{$u$}}
                    \put(239,95){\makebox(0,0)[r]{$v_q$}}
                    \put(239,35){\makebox(0,0)[r]{$v_1$}}
                    \put(250,35){\circle{7}}
                    \put(250,115){\circle{7}}
                    \put(250,65){\oval(16,80)}
                \end{picture}
                \caption{The three possible stopping conditions of the constructed path, which is encircled, from the proof of Theorem~\ref{theorem:branching_k-path_vertex_cover_rooted_directed_forest_lower_bound}. The vertices belonging to $P$ are marked with separate circles around them.}
                \label{figure:directed_constructed_path_and_3_stopping_conditions}
            \end{figure}
                        
            If there exists a leaf $v_1$ for which we stopped because $v_q$ is a root in $F$ or because $u$ belongs to $P$ and is not a branching vertex, then we remove the vertices $v_1$, \ldots, $v_q$ from $F$ and denote the resulting rooted forest on $n-q$ vertices by $H$. Since $q \leq k$ and $n \geq k+1$, we have $n-q \geq (k+1)-k=1$, which means that $H$ is not empty. Denote by $Q$ the restriction of $P$ to the vertices of $H$, that is, $Q=P \cap V(H)$. Notice that exactly $1$ of the removed vertices $v_1$, \ldots, $v_q$ belongs to $P$, namely, the leaf $v_1$. Therefore, we have $P=Q \cup \{v_1\}$ and $|P|=|Q|+1$. If we stopped because $v_q$ is a root in $F$, then all leaves in $H$ are also leaves in $F$ and thus belong to $P$. If we stopped because $u$ belongs to $P$ and is not a branching vertex, then $u$ becomes a leaf in $H$ that belongs to $P$, while all other leaves in $H$ are also leaves in $F$ and thus belong to $P$. So, in both cases, all leaves of $H$ belong to $P$, and thus to $Q$. Also, it is easy to see that every vertex in $H$ that is branching in $F$ is also branching in $H$. Indeed, the only possible vertex in $H$ whose out-degree decreases after the removal of the vertices $v_1$, \ldots, $v_q$ is $u$ (if $v_q$ is not a root and $u$ exists). But $u$ is not branching in $F$ by the assumption of the current case. All other vertices in $H$ preserve their out-degree after the removal, and hence those of them that were branching in $F$ remain branching in $H$. Take an arbitrary directed path in $H$ on $k$ vertices. Obviously, it is also a directed path in $F$. Therefore, it contains either a vertex from $P$ or a branching vertex in $F$. If it contains a vertex from $P$, then that vertex also belongs to $Q$. If the path contains a branching vertex in $F$, then that branching vertex is also a branching vertex in $H$. So, $Q$ is a branching $k$-path vertex cover of $H$. Therefore, we can apply the inductive hypothesis to $H$ and get that $|Q| \geq \frac{(n-q)+k}{2k}$. Now, we have $|P|=|Q|+1 \geq \frac{(n-q)+k}{2k}+1=\frac{n+k}{2k}+\frac{1}{2}+\frac{k-q}{2k} \geq \frac{n+k}{2k}+\frac{1}{2}>\frac{n+k}{2k}$, as required.
                    
            The only remaining case is when, for each starting leaf, we stopped because $u$ is a branching vertex (regardless of whether $u$ also belongs to $P$ or not). Then we choose a leaf $v_1$ for which the branching vertex $u$ is at the largest distance from the root of its connected component. Consider all directed paths from $u$ to leaves (not assuming that they are constructed by the described process for the leaves). Denote by $b$ the number of these directed paths (and it is equal to the number of leaves at the ends of these directed paths). Since $u$ is a branching vertex, we have $b \geq 2$. If any of these $b$ directed paths contains branching vertices other than $u$, then take the subpath from the next vertex after the last branching vertex $u'$ on that directed path to a leaf $v_1'$, and that would be the constructed directed path for $v_1'$ with the distance from $u'$ to the root larger than that from $u$ to the root (because the unique path from the root to $u'$ passes through $u$), which would contradict the choice of the leaf $v_1$. Therefore, the directed paths from $u$ to leaves do not contain branching vertices other than $u$. Since the currently considered case assumes that every constructed directed path for every leaf stopped immediately before a branching vertex, and we just proved that there are no other branching vertices except $u$ on any of the considered paths from $u$ to leaves, the constructed directed paths for those leaves must stop exactly immediately before $u$. In other words, if we remove $u$ from these directed paths from $u$ to leaves, then we get exactly the constructed directed paths for these leaves. We remove the vertices of these constructed directed paths from $F$ (keeping the branching vertex $u$) and denote the resulting rooted forest by $H$ and the number of vertices in it by $n'$. We removed $b$ constructed directed paths, each of which contains at most $k$ vertices. So, in total, we removed at most $bk$ vertices. Therefore, we have $n' \geq n-bk$. Since the branching vertex $u$ was not removed, $H$ is not empty and $n' \geq 1$. Denote by $Q$ the restriction of $P$ to the vertices of $H$ with an added vertex $u$, if it is not already in $P$, that is, $Q=(P \cap V(H)) \cup \{u\}$. From the property proved earlier that every constructed path contains exactly $1$ vertex belonging to $P$ and that vertex is the starting leaf, we obtain that exactly $b$ of the removed vertices belong to $P$, namely, the leaves. Therefore, we have either $|P|=|Q|+b$ or $|P|=|Q|+b-1$, depending on whether $u$ belongs to $P$. In both cases, we have $|P| \geq |Q|+b-1$. Let us prove that $Q$ is a branching $k$-path vertex cover of $H$. After the removal of the vertices, $u$ becomes a leaf in $H$, and, by definition, it belongs to $Q$. All other leaves in $H$ are also leaves in $F$ and thus belong to $P$ and to $Q$. Take an arbitrary directed path in $H$ on $k$ vertices. Obviously, it is also a directed path in $F$. Therefore, it contains either a vertex from $P$ or a branching vertex in $F$. If it contains a vertex from $P$, then that vertex also belongs to $Q$. If the path contains a branching vertex in $F$ that is not $u$, then that branching vertex is also a branching vertex in $H$. If the path contains a branching vertex that is $u$, then it contains a vertex from $Q$ because $u$ belongs to $Q$. In all cases, that path contains either a vertex from $Q$ or a branching vertex in $H$. So, $Q$ is a branching $k$-path vertex cover of $H$. Therefore, we can apply the inductive hypothesis to $H$ and get that $|Q| \geq \frac{n'+k}{2k}$. Now, we have $|P| \geq |Q|+b-1 \geq \frac{n'+k}{2k}+b-1 \geq \frac{(n-bk)+k}{2k}+b-1=\frac{n+k}{2k}+\frac{b-2}{2} \geq \frac{n+k}{2k}$, as required.

            \begin{figure}[htb]
                \centering
                \setlength{\unitlength}{0.4mm}
                \begin{picture}(105,108)(45,33)
                    \put(60,135){\circle*{3}}
                    \put(60,115){\circle*{3}}
                    \put(60,95){\circle*{3}}
                    \put(60,75){\circle*{3}}
                    \put(60,55){\circle*{3}}
                    \put(60,35){\circle*{3}}
                    
                    \put(60,135){\vector(0,-1){19}}
                    \put(60,115){\vector(0,-1){19}}
                    \put(60,95){\vector(0,-1){19}}
                    \put(60,75){\vector(0,-1){19}}
                    \put(60,55){\vector(0,-1){19}}

                    \put(55,135){\makebox(0,0)[r]{$v_1$}}
                    \put(55,95){\makebox(0,0)[r]{$v_k$}}
                    \put(55,35){\makebox(0,0)[r]{$v_{2k}$}}
                    
                    \put(60,95){\vector(4,-1){58.8}}
                    
                    \put(120,65){\oval(60,60)}
                    \put(120,80){\circle*{3}}
                    \put(120,80){\vector(0,-1){19}}
                    \put(135,50){\makebox(0,0){$F_i$}}
                \end{picture}
                \caption{The rooted directed forest $F_{i+1}$ from the tightness part of the proof of Theorem~\ref{theorem:branching_k-path_vertex_cover_rooted_directed_forest_lower_bound}, illustrated for $k=3$.}
                \label{figure:directed_tightness_construction}
            \end{figure}
                        
            To show that the lower bound is tight, we recursively construct a sequence of rooted directed forests $\{F_i\}$ on which the lower bound is attained. Define $F_1$ to be the rooted directed forest consisting of a directed path on $k$ vertices. For each $i \geq 1$, define $F_{i+1}$ to be a directed path $v_1 \ldots v_{2k}$ on $2k$ vertices with a copy of $F_i$ attached to the vertex $v_k$ by an additional arc from that vertex to the root of the copy of $F_i$, as shown in Figure~\ref{figure:directed_tightness_construction}. Denote by $n_i$ the number of vertices in $F_i$. By construction, we have $n_1=k$ and $n_{i+1}=n_i+2k$. So, using induction, we derive that $n_i=k(2i-1)$. Also, by construction, $F_1$ has a single leaf and $F_{i+1}$ has one more leaf than $F_i$. So, using induction, we derive that $F_i$ has exactly $i$ leaves.
            
            Now, we prove by induction on $i$ that every directed path on $k$ vertices in $F_i$ contains either a leaf or a branching vertex. For $i=1$, there is only a single directed path on $k$ vertices in $F_1$, consisting of the entire $F_1$, and it contains the leaf. Suppose that the claim holds for $F_i$ and consider it for $F_{i+1}$. Consider an arbitrary directed path on $k$ vertices in $F_{i+1}$. We separately consider three cases of where that path lies. If that path is entirely contained in the copy of $F_i$, then, by the induction hypothesis, it contains either a leaf of $F_i$ or a branching vertex of $F_i$. Since the out-degrees of the vertices in the copy of $F_i$ remain unchanged when forming $F_{i+1}$ (the only new arc incident to those vertices is incoming into the root of $F_i$ and there are no outgoing new arcs), any leaf in $F_i$ is also a leaf in $F_{i+1}$ and any branching vertex in $F_i$ is also a branching vertex in $F_{i+1}$. Therefore, the considered path contains either a leaf of $F_{i+1}$ or a branching vertex of $F_{i+1}$, as required. If the considered path is entirely contained within the path $v_1 \ldots v_{2k}$, then either it contains $v_k$, which is a branching vertex, or it is exactly the path $v_{k+1} \ldots v_{2k}$, which contains the leaf $v_{2k}$. The last remaining case is when the considered path contains vertices both from the path $v_1 \ldots v_{2k}$ and from the copy of $F_i$. Since the only arc connecting these two parts is from $v_k$ to the root of $F_i$, the considered path must contain $v_k$, which is a branching vertex. So, in all cases, the considered path contains either a leaf or a branching vertex, and the claim holds for $F_{i+1}$.
            
            Take the set $P_i$ to be the set of all $i$ leaves in $F_i$. The claim that we have just proved implies that $P_i$ is a branching $k$-path vertex cover of $F_i$. We have $|P_i|=i$ and $\frac{n_i+k}{2k}=\frac{k(2i-1)+k}{2k}=\frac{2ki}{2k}=i$. Therefore, we have $|P_i|=\frac{n_i+k}{2k}$, which means that the lower bound is attained on $F_i$.
        \end{proof}

    \section{The undirected case}
        \begin{theorem}\label{theorem:branching_k-path_vertex_cover_undirected_forest_lower_bound}
            Let $F$ be an undirected forest on $n \geq 2$ vertices. Let $k \geq 2$ be a natural number. Then $\psi_b(F,k) \geq \frac{n+3k-1}{2k}$. That lower bound is tight (when the expression in the lower bound is an integer).
        \end{theorem}
        \begin{proof}
            Let $P$ be an arbitrary branching $k$-path vertex cover of $F$. Denote the number of connected components of $F$ by $p$. Clearly, we have $p \geq 1$. If a connected component of $F$ consists of a single isolated vertex, then this vertex must belong to $P$, and we remove this vertex. In each connected component of $F$ containing at least $2$ vertices, we choose an arbitrary leaf $u$ (it is well-known that a tree with at least $2$ vertices always contains a leaf). Denote its unique neighbor by $v$. We remove $u$ (along with the edge $uv$), make $v$ the root, and orient all remaining edges of that connected component away from the new root $v$. When this is done to all connected components, denote the resulting rooted directed forest by $H$. Clearly, the number of vertices in $H$ is $n-p$.

            \begin{figure}[htb]
                \centering
                \setlength{\unitlength}{0.4mm}
                \begin{picture}(243,63)(-1.5,8.5)
                    \put(50,50){\circle*{3}}
                    \put(30,70){\circle*{3}}
                    \put(50,70){\circle*{3}}
                    \put(30,30){\circle*{3}}
                    \put(15,10){\circle*{3}}
                    \put(45,10){\circle*{3}}
                    \put(70,30){\circle*{3}}
                    \put(55,10){\circle*{3}}
                    \put(85,10){\circle*{3}}
                    \put(15,50){\circle*{3}}
                    \put(0,65){\circle*{3}}
                    \put(0,35){\circle*{3}}
                    \put(70,70){\circle*{3}}
                    \put(90,70){\circle*{3}}

                    \Line(50,50)(30,70)
                    \Line(50,50)(50,70)
                    \Line(50,50)(30,30)
                    \Line(30,30)(15,10)
                    \Line(30,30)(45,10)
                    \Line(50,50)(70,30)
                    \Line(70,30)(55,10)
                    \Line(70,30)(85,10)
                    \Line(30,30)(15,50)
                    \Line(15,50)(0,65)
                    \Line(15,50)(0,35)
                    \Line(50,50)(70,70)
                    \Line(70,70)(90,70)

                    \put(26,70){\makebox(0,0)[r]{$u$}}
                    \put(56,50){\makebox(0,0)[l]{$v$}}

                    \put(105,40){\vector(1,0){30}}

                    \put(200,50){\circle*{3}}
                    \put(200,70){\circle*{3}}
                    \put(180,30){\circle*{3}}
                    \put(165,10){\circle*{3}}
                    \put(195,10){\circle*{3}}
                    \put(220,30){\circle*{3}}
                    \put(205,10){\circle*{3}}
                    \put(235,10){\circle*{3}}
                    \put(165,50){\circle*{3}}
                    \put(150,65){\circle*{3}}
                    \put(150,35){\circle*{3}}
                    \put(220,70){\circle*{3}}
                    \put(240,70){\circle*{3}}

                    \put(200,50){\vector(0,1){19}}
                    \put(200,50){\vector(-1,-1){19.2}}
                    \put(180,30){\vector(-3,-4){14.2}}
                    \put(180,30){\vector(3,-4){14.2}}
                    \put(200,50){\vector(1,-1){19.2}}
                    \put(220,30){\vector(-3,-4){14.2}}
                    \put(220,30){\vector(3,-4){14.2}}
                    \put(180,30){\vector(-3,4){14.2}}
                    \put(165,50){\vector(-1,1){14.2}}
                    \put(165,50){\vector(-1,-1){14.2}}
                    \put(200,50){\vector(1,1){19.2}}
                    \put(220,70){\vector(1,0){19}}

                    \put(206,50){\makebox(0,0)[l]{$v$ (root)}}
                \end{picture}
                \caption{An example of the transformation of a connected component of $F$ (left) into a connected component of $H$ (right) from the proof of Theorem~\ref{theorem:branching_k-path_vertex_cover_undirected_forest_lower_bound}.}
                \label{figure:undirected_to_directed_reduction_transformation_example}
            \end{figure}

            If $H$ is empty, then all connected components of $F$ are isolated vertices, and they all belong to $P$. So, we have $|P|=n$. It remains to verify that $n \geq \frac{n+3k-1}{2k}$. That inequality is equivalent to $\frac{3}{2}+\frac{1}{2(2k-1)} \leq n$. For $k \geq 2$, we have $\frac{3}{2}+\frac{1}{2(2k-1)} \leq \frac{3}{2}+\frac{1}{2(2 \cdot 2-1)}=\frac{3}{2}+\frac{1}{6}=\frac{5}{3}<2 \leq n$. So, the required inequality holds, which completes the proof for the case of empty $H$. In what follows, we assume that $H$ is not empty, that is, $n-p \geq 1$.

            Denote by $Q$ the restriction of $P$ to the vertices of $H$, that is, $Q=P \cap V(H)$. It is easy to see that all leaves in $H$ are also leaves in $F$, and thus they belong to $Q$.
            
            Consider a vertex $x$ of $H$ that is a branching vertex in $F$. If $x$ is the new root $v$ in one of the connected components, then it has degree at least $3$ in $F$, and hence at least $2$ neighbors in $F$ other than $u$. The arcs in $H$ from $v$ to these two neighbors are oriented away from $v$. Therefore, $v$ has out-degree at least $2$ in $H$, which means that it is a branching vertex in $H$. If $x$ is not the new root $v$ of its connected component, then we consider the unique directed path from $v$ to $x$ in $H$. This path has exactly $1$ arc incident to $x$, and it is incoming into $x$. Since $x$ is a branching vertex in $F$, it has at least $2$ other edges incident to it. These two edges must be oriented away from $x$ because they extend the unique path from $v$ to $x$. Therefore, $x$ has out-degree at least $2$ in $H$, which means that it is a branching vertex in $H$. So, in both cases, $x$ is a branching vertex in $H$. This means that every vertex of $H$ that is a branching vertex in $F$ is also a branching vertex in $H$.

            Consider a directed path in $H$ on $k$ vertices. Clearly, its vertices form a path in $F$. Therefore, it contains either a branching vertex in $F$ or a vertex belonging to $P$. If it contains a branching vertex in $F$, then, from the property that we proved above, this vertex is also a branching vertex in $H$. If it contains a vertex belonging to $P$, then this vertex also belongs to $Q$. This means that $Q$ is a branching $k$-path vertex cover of $H$. By Theorem~\ref{theorem:branching_k-path_vertex_cover_rooted_directed_forest_lower_bound}, we have $|Q| \geq \frac{(n-p)+k}{2k}$.
            
            Since the removed vertex (an isolated vertex or a leaf) in each connected component belongs to $P$, we have $|Q|=|P|-p$. Therefore, we have $|P|=|Q|+p \geq \frac{(n-p)+k}{2k}+p=\frac{n+k+p(2k-1)}{2k} \geq \frac{n+k+1 \cdot (2k-1)}{2k}=\frac{n+3k-1}{2k}$, as claimed.

            \begin{figure}[htb]
                \centering
                \setlength{\unitlength}{0.4mm}
                \begin{picture}(127,64)(40,18)
                    \put(70,50){\oval(60,60)}
                    \put(55,65){\makebox(0,0){$F_i$}}
                    \put(90,50){\circle*{3}}
                    \Line(75,50)(88,50)
                    \put(90,57){\makebox(0,0){$u$}}
                    
                    \put(115,60){\circle*{3}}
                    \put(140,70){\circle*{3}}
                    \put(165,80){\circle*{3}}
                    \Line(90,50)(115,60)
                    \Line(115,60)(140,70)
                    \Line(140,70)(165,80)
                    
                    \put(115,40){\circle*{3}}
                    \put(140,30){\circle*{3}}
                    \put(165,20){\circle*{3}}
                    \Line(90,50)(115,40)
                    \Line(115,40)(140,30)
                    \Line(140,30)(165,20)
                \end{picture}
                \caption{The forest $F_{i+1}$ from the tightness part of the proof of Theorem~\ref{theorem:branching_k-path_vertex_cover_undirected_forest_lower_bound}, illustrated for $k=3$.}
                \label{figure:undirected_tightness_construction}
            \end{figure}

            To show that the lower bound is tight, we recursively construct a sequence of forests $\{F_i\}$ on which the lower bound is attained. Define $F_1$ to be the forest consisting of a path on $k+1$ vertices. For each $i \geq 1$, take an arbitrary leaf $u$ of $F_i$ and add two new paths, each of which is on $k$ new vertices and is attached to $u$ by an additional edge, as shown in Figure~\ref{figure:undirected_tightness_construction}. Denote the resulting forest by $F_{i+1}$. Denote by $n_i$ the number of vertices in $F_i$. By construction, we have $n_1=k+1$ and $n_{i+1}=n_i+2k$. So, using induction, we derive that $n_i=k(2i-1)+1$. Now, we calculate the number of leaves in $F_i$. By construction, $F_1$ has $2$ leaves. When forming $F_{i+1}$ from $F_i$, $2$ new leaves are added in the new attached paths, and the vertex $u$ converts from a leaf into a branching vertex. Therefore, in total, $F_{i+1}$ has one more leaf than $F_i$. So, using induction, we derive that $F_i$ has exactly $i+1$ leaves.
            
            Now, we prove by induction on $i$ that every path on $k$ vertices in $F_i$ contains either a leaf or a branching vertex. For $i=1$, there are exactly $2$ paths on $k$ vertices in $F_1$, and both of them contain a leaf. Suppose that the claim holds for $F_i$ and consider it for $F_{i+1}$. Consider an arbitrary path on $k$ vertices in $F_{i+1}$. We separately consider three cases of where that path lies. If that path is entirely contained in the copy of $F_i$, then, by the induction hypothesis, it contains either a leaf of $F_i$ or a branching vertex of $F_i$. We observe that the only vertex in the copy of $F_i$ that changes its degree is $u$, and it converts from a leaf into a branching vertex. All other leaves in $F_i$ remain leaves in $F_{i+1}$, and all branching vertices in $F_i$ remain branching vertices in $F_{i+1}$. Therefore, the considered path contains either a leaf of $F_{i+1}$ or a branching vertex of $F_{i+1}$, as required. If the considered path is entirely contained within one of the two new paths on $k$ vertices attached to $u$, then it comprises this entire new path and hence it contains the leaf at the end of it. The last remaining case is when the considered path contains vertices from at least two of the three parts: the copy of $F_i$ and the two new paths attached to $u$. Since these three parts are connected only through the vertex $u$, the considered path must contain $u$, which is a branching vertex. So, in all cases, the considered path contains either a leaf or a branching vertex, and the claim holds for $F_{i+1}$.
            
            Take the set $P_i$ to be the set of all $i+1$ leaves in $F_i$. The claim that we have just proved implies that $P_i$ is a branching $k$-path vertex cover of $F_i$. We have $|P_i|=i+1$ and $\frac{n_i+3k-1}{2k}=\frac{(k(2i-1)+1)+3k-1}{2k}=\frac{2k(i+1)}{2k}=i+1$. Therefore, we have $|P_i|=\frac{n_i+3k-1}{2k}$, which means that the lower bound is attained on $F_i$.
        \end{proof}

    \printbibliography

@article{tu_survey_k-path_vertex_cover,
    author = {Tu, Jianhua},
    title = {A survey on the $k$-path vertex cover problem},
    journal = {Axioms},
    volume = {11},
    year = {2022},
    number = {5},
    pages = {191},
    doi = {10.3390/axioms11050191}
}
\end{document}